\documentclass[11pt,twoside]{article%
}\usepackage{epsfig,amsthm,
amsmath,amsfonts,subeqnarray
}
\date{}
\usepackage{verbatim,comment,hyperref}
\title{
F. John's stability conditions vs. A. Carasso's SECB constraint for backward
parabolic problems\footnote{To appear in \underline{Inverse Problems}}
}
\author{
Jinwoo Lee\thanks{Department of Mathematics
        Kwangwoon University,
        Seoul 139-701
        Korea.
({\tt jinwoolee@kw.ac.kr}).
} \and 
Dongwoo Sheen\thanks{Department of Mathematics and Interdisciplinary Program
  in Computational Science \& Technology,
Seoul National University, Seoul 151-747, Korea.
({\tt  sheen@snu.ac.kr}). 
}}

\renewcommand{\tilde}{\widetilde}
\newcommand{\rb}[1]{\raisebox{1.5ex}[0pt]{#1}}
\newcommand{\be}{\begin{eqnarray}}
\newcommand{\ee}{\end{eqnarray}}
\newcommand{\bal}{\begin{aligned}}
\newcommand{\eal}{\end{aligned}}
\newcommand{\bes}{\begin{eqnarray*}}
\newcommand{\ees}{\end{eqnarray*}}
\newcommand{\bs}{\begin{subeqnarray}}
\newcommand{\es}{\end{subeqnarray}}
\newcommand{\bss}{\begin{subeqnarray*}}
\newcommand{\ess}{\end{subeqnarray*}}

\newtheorem{theorem}{Theorem}[section]
\newtheorem{lemma}[theorem]{Lemma}
\newtheorem{proposition}[theorem]{Proposition}
\newtheorem{remark}[theorem]{Remark}
\newtheorem{corollary}[theorem]{Corollary}
\newtheorem{definition}[theorem]{Definition}

\newcommand{\defref}[1]{Definition~\ref{#1}}
\newcommand{\thmref}[1]{Theorem~\ref{#1}}
\newcommand{\secref}[1]{Section~\ref{#1}}

\newcommand{\lemref}[1]{Lemma~\ref{#1}}
\newcommand{\propref}[1]{Proposition~\ref{#1}}
\newcommand{\rmkref}[1]{Remark~\ref{#1}}

\newcommand{\tabref}[1]{Table~\ref{#1}}

\numberwithin{equation}{section}

\def\for{\quad{for}\ }
\def\with{\quad{with}\ }

\def\p{\partial}

\def\la{\lambda}

\def\Om{\Omega}
\def\O{\Omega}

\def\de{\delta}

\def\Ga{\Gamma}
\def\G{\Gamma}

\def\Re{\operatorname{Re}}

\def\and{\quad\text{and}\quad}
\def\i{\text{\rm i}}

\begin{document}
\maketitle
\begin{abstract}
In order to solve backward parabolic problems F. John [{\it
  Comm. Pure. Appl. Math.} (1960)]
introduced the two constraints
``$\|u(T)\|\le M$'' and  $\|u(0) - g \| \le \delta$ where $u(t)$
satisfies the backward heat equation for $t\in(0,T)$ with the initial data $u(0).$

The {\it slow-evolution-from-the-continuation-boundary} (SECB) constraint
has been introduced by A. Carasso in [{\it SIAM J. Numer. Anal.} (1994)] 
to attain continuous dependence on data for 
backward parabolic problems even at the continuation boundary $t=T$. 
The additional ``SECB constraint'' guarantees a significant improvement in
stability up to $t=T.$ 
In this paper we prove that the same type of stability can be obtained by using only
two constraints among the three. More precisely, we show that 
the a priori boundedness condition $\|u(T)\|\le M$ is redundant.
This implies that the Carasso's SECB condition can be used to replace the
a priori boundedness condition of F. John with an improved stability estimate.
Also a new class of regularized solutions is introduced for backward parabolic
problems with an SECB constraint. The new regularized solutions are optimally stable
and we also provide a constructive scheme to compute.
Finally numerical examples are provided.
\end{abstract}

\noindent{\bf Keywords:} slow evolution constraint (SECB),
backward parabolic problem, ill-posed problem, 
Laplace transform

\noindent{\bf AMS Subject Classes:} 47A52, 44A10, 65M30

\section{Introduction}
One of the most classical inverse and ill-posed problems \cite{hadamard, john60, payne}
is to find the past heat distribution $u(\cdot, t)$ for $0\le s < T$
based on the temperature distribution $u(\cdot, T)$ known at the current time $T$,
which is formulated as to find $u$ such that
\bes
\frac{\p u}{\p s} - \Delta u &=& 0  \quad\text{on }\O\times (0,T); \\
u(x,s)&=&0\quad\forall (x,s)\in \p\O\times (0,T); \quad u(x,T)=g(x) \quad\forall x\in
\O.
\ees
The change of variables $T-s = t$ yields to the problem
\begin{subeqnarray}\label{eq:problem0}
\frac{\p u}{\p t} + \Delta u &=& 0\quad\text{on } \O\times (0,T); \slabel{eq:problem0a} \\
u(x,t)&=&0\quad\forall (x,t)\in \p\O\times (0,T); \quad u(x,0)=g(x) \quad\forall x\in
\O.\slabel{eq:problem0b}
\end{subeqnarray}
Let $-A$ be a second-order linear uniformly elliptic partial 
differential operator 
on a domain $\Om$ with the {\it homogeneous} Dirichlet boundary condition on
the boundary $\p\O$, which is
a self-adjoint operator in $L^2(\O)$, such that $-A$ has the eigenvalues
$0<\lambda_1\le \lambda_2\le \cdots <\infty$ with corresponding orthonormal
eigenfunctions
$\phi_n$'s.
Problem \eqref{eq:problem0} generalizes to
\begin{subeqnarray}\label{eq:problem}
\frac{\p u}{\p t} + Au &=& 0  \quad\text{on }\O\times (0,T),\\
 u(x,0)&=&g(x) \quad\forall x\in \O.
\end{subeqnarray}
The problem \eqref{eq:problem} is ill-posed in the sense that 
the solution does not depend continuously on the data $g$ \cite{hadamard, john60, payne}.
To stabilize it, F. John \cite{john60} introduced a fundamental concept to
prescribe a bound on the solution at $t=T$ with relaxation of the initial data $g.$
More precisely, given positive constants $M$ and $\de$, consider the class of
solutions $u_j$'s which satisfy 
\begin{subeqnarray}\label{john-stab-sol}
\frac{\p u_j}{\p t} + A u_j &=& 0 \quad\text{on }\O\times (0,T),\slabel{eq:backward} \\
\|u_j(0) - g \| \le \de&\quad& \mbox{and } \qquad
\|u_j(T)\| \le M. \slabel{condition1}
\end{subeqnarray}
Then $u_j$'s satisfy the following H\"older-type stability
\cite{agmon,payne}: for any two solutions $u_j, j =1, 2,$ of
\eqref{john-stab-sol},
\be
\|u_1(t)-u_2(t)\|\le 2 M^{t/T} \de^{1-t/T}
\for t\in [0,T].
\label{eq:holder}
\ee
where $\|\cdot\|$ denotes the $L^2(\O)$-norm.
The $\de$ on the right hand side of \eqref{eq:holder} guarantees
continuous dependence on data for $t\in (0,T)$.

However, one loses the continuous dependence on data property 
at the continuation boundary ($t=T$) no matter how small a $\delta$ is chosen,
which has been annoying mathematicians and scientists for about three decades since
F. John's work \cite{john60}.
To overcome it, Carasso in his seminal works \cite{carasso, car_nonsmooth} 
introduced an additional constraint, called a
{\it slow evolution from the continuation boundary} (SECB) constraint, which
is an a priori statement about the rate of change of the solution near the continuation boundary.
The definition of SECB constraint is discussed in \eqref{eq:s_star} and
\defref{def:secb} in \S 2. With an extra SECB constraint, $u_j$'s fulfill the following
improved stability over \eqref{eq:holder}:
\be\label{eq:secb1}
\|u_1(t)-u_2(t)\|\le 2\Lambda^{t/T}\de, \quad  t\in [0,T],
\ee
where $\Lambda<\frac{M}{\de}$ is a positive constant.
The three constraints, namely \eqref{condition1} and an SECB constraint,
have been used extensively and have proved usefulness for stabilizing ill-posed
problems \cite{car_lin_nonlin,car_logarithm, car_blind,car_preceed}.
In \cite{carasso} Carasso also provides a constructive scheme to find regularized solutions
which can be implemented when $A$ in \eqref{eq:problem} has constant
coefficients.


In this paper optimal stability \eqref{eq:secb1} is proved by using the two
conditions $\|u_j(0)-g\|\le \de$ and an SECB constraint only. In other words,
{\it we show that an a priori bound $\|u_{j}(T)\|\le M$ in \eqref{condition1} is redundant.}
Also a class of new regularized solutions is introduced, which gives an 
optimal stability of the form \eqref{eq:secb1},
and can be obtained numerically even when $A$ has variable
coefficients,  
which will make the SECB constraints more useful and practical
in many application areas.
Applications of the above idea to image deblurring will be available in a
forthcoming paper \cite{jleesheen-denoising}.

The rest of the paper is organized as follows. \secref{sec2}
reviews the concept of SECB and its properties. The new proof of the stability
is then given in \secref{sec3}. In \secref{sec4} a new class of regularized 
solutions is defined and its optimal stability is proved.
In \secref{sec5} numerical results are reported with the proposed
constructive algorithm.

\section{Slow evolution from the continuation boundary (SECB)}\label{sec2}
In this section the notion of SECB and its properties are reviewed in brief.
More detailed explanations and extensive applications of SECB can be found in
\cite{carasso, car_nonsmooth,car_lin_nonlin,car_logarithm, car_blind}.

Let us begin with the following simple observation.
For given positive $K,M,$ and $\de>0$ with $0<K\ll \frac{M}{\de},$ 
set $s^*=s^*(\de,M,K)$ to be 
\be\label{eq:s_star}
s^*(\de,M,K)=T\frac{\log(\frac{M}{\de}-K)}{\log(\frac{M}{\de})}.
\ee
so that
$\frac{M}{\de} = K+ \left(\frac{M}{\de}\right)^{s^*/T}$.
Set $g(x,y)\equiv 0.$
Then $u_1(x,y,t)\equiv 0$ is a (trivial) solution to \eqref{eq:problem},
and to \eqref{john-stab-sol}.
Let $u_2(x,y,t)$ be another solution to \eqref{john-stab-sol}
with an additional constraint
$\|u_1(t)-u_2(t)\| = \|u_2(t)\|= M^{t/T}\de^{1-t/T}$; for example, 
$u_2(x,y,t)= \de e^{\la t}\phi(x,y)$ satisfies such an additional condition
provided 
$\la=\frac1{T}\log(\frac{M}{\de})$, where $\phi(x,y)$ 
is an orthonormal eigenfunction of the spatial operator $-A$ in 
\eqref{john-stab-sol} with eigenvalue $\la$, 
which mimics one of the worst case solutions to \eqref{eq:backward}
with the constraints \eqref{condition1}.
Then $\|u_2(T)-u_2(s^*)\|=\|M\phi(x,y) - (M-\de K)\phi(x,y)\|= K\de$;
moreover, $\|u_2(T)-u_2(t)\|\le K \de$ if and only if $t\ge s^*$.
Thus if an extra constraint on $u_j(x,y,t)$ is imposed such that
$\|u_j(T)-u_j(s)\|\le K\de$ for some $s$ which is {\it less than} $s^*$, 
a better estimation than that given in \eqref{eq:holder} is expected. 
This observation leads to the following definition, firstly appeared in \cite{carasso},

\begin{definition}\label{def:secb}{\rm[Carasso (1994)]}
For given $K>0$, let $s^*$ be defined by \eqref{eq:s_star}. If 
there exists a known fixed $s>0$ with $s<s^*$ such that
\be\label{eq:secb}
\|u_j(T)-u_j(s)\| \le K\de,
\ee
$u_j$ is said to satisfy ``slow evolution from the continuation 
boundary(SECB)'' constraint.
\end{definition}
\begin{remark}
Condition \eqref{eq:secb} implies that the class of solutions is restricted to 
satisfy the {\bf slow evolution} condition near the continuation boundary $t=T.$
\end{remark}

Carasso then proves that any two solutions to \eqref{eq:backward} with constraints 
\eqref{condition1}
and \eqref{eq:secb} have the following improved stability:
\begin{theorem}{\rm[Carasso (1994)]}\label{thm:stabil2}
Let $u_j(t)$, $j=1,2$, be two solutions to \eqref{eq:backward} with constraints
\eqref{condition1} and \eqref{eq:secb}. Then
\be\label{eq:stabil2}
\|u_1(t)-u_2(t)\|\le 2\Lambda^{t/T}\de, \quad t\in [0,T],
\ee
where $\Lambda=\Lambda(K,s)$ is the unique root of the equation 
\be\label{eq:alge}
x=K+x^{s/T}.
\ee
\end{theorem}
As for $\Lambda$, we have the following estimation.
\begin{lemma}\label{lemma:gamma_ineq}
Given $\de$, $M$, and $K$ satisfying $0<\de\ll M$,
$0<K\ll \frac{M}{\de}$, and $K+1<\frac{M}{\delta}$, let $s^*$ be defined by \eqref{eq:s_star}. 
For $0<s<s^*$, let $\Lambda=\Lambda(K,s)$ be the unique root
of \eqref{eq:alge}. Then
\be\label{eq:gamma_ineq}
K+1<\Lambda< \frac{M}{\de}.
\ee
Moreover, with any $z_1>0$, the iterates
$z_{n+1}=K+z_n^{s/T}, n=1,2\cdots,$ converge to $\Lambda$.
\end{lemma}
\begin{proof}
Since $\frac{M}{\de}$ is the root of \eqref{eq:alge} with $s=s^*$ 
by the definition of $s^*$, 
and the root of \eqref{eq:alge}
decreases monotonically with decreasing $s$, we have $\Lambda<\frac{M}{\de}$.
We also have $K+1<\Lambda$ since $1$ is the root of \eqref{eq:alge}
with $K=0$, and $K>0$. The last statement is a standard result
of the fixed point iteration.
\end{proof}
\begin{remark}
Notice that the second statement of \lemref{lemma:gamma_ineq} 
is slightly different
from that given in \cite{carasso} which states that $K+1\le z_1\le \frac{M}{\de}$.
This will play an important role (see \rmkref{bcd}) in our analysis.
\end{remark}
Since $\Lambda$ is less than $\frac{M}{\de}$ by \lemref{lemma:gamma_ineq},  
\thmref{thm:stabil2} shows an improved stability estimate compared to
\eqref{eq:holder}. Moreover, this retains its continuous dependence on data
even at the continuation boundary $t=T$. 

\section{$\|u_j(T)\|\le M$ is redundant}\label{sec3}
In this section we prove \eqref{eq:stabil2} by using 
one condition $\|u_{j}(0)-g\|\le \de, j=1,2,$ in \eqref{condition1}
and the SECB condition \eqref{eq:secb} only. 
\begin{theorem}\label{thm:secb}
For given data $g\in L^2(\O)$, let $u_j(t)$, $j=1,2$, 
be two solutions to 
\[
\frac{\p u_j}{\p t} + Au_j = 0 \quad\text{on }\O\times (0,T)
\]
with constraints
\be\label{cond:secb}
\|u_j(0)-g\|\le \de  \quad\mbox{and}\quad
\|u_j(T)-u_j(s)\| \le K\de
\ee
for known positive parameters $\de>0$, $K>0$, and $s\in(0,T)$. 
Then
\be\label{eq:thm_diff}
\|u_1(t)-u_2(t)\|\le 2\Lambda^{t/T}\de, \quad t\in [0,T],
\ee
where $\Lambda=\Lambda(K,s)$ is the unique root of the equation 
\eqref{eq:alge}.
\end{theorem}

In order to prove the above theorem, we will need the following 
preliminary result to bound $\|z(s)\|$:
\begin{lemma}\label{lem:bound}
Let $z(t)$ be the difference of two solutions $u_j(t),j=1,2,$ to \eqref{eq:backward}
with constraints \eqref{cond:secb}. Then, for $0<s<T,$ we have
\be\label{eq:final}
\|z(s)\|^2
\le \left(\frac{2K{\de}}{T-s}\right)^2 +
\sum_{n=1}^{l} z_n^2(1-\la_n^2)e^{2\la_n s}
\ee
where $z_n= (z(0),\phi_n)$ with the standard $L^2(\Omega)$-inner product notation $(\cdot,\cdot)$ and $l\ge 0$ and
$l\ge 0$ is the largest integer such that $\la_l< 1.$
\end{lemma}
\begin{proof}
Since $z(t)$ is the solution to \eqref{eq:problem} with initial data
$z(0)=u_1(0)-u_2(0)$, it admits the following
representation\footnote{Notice that the assumption on the
  existence of the solution $u_j(t)$ for $t\in (0,T]$ implies that $\|u_j(T)\| \le M_j$
with a possibly different bound $M_j>0$ for each $j.$
Thus $\|\sum_{n=1}^{\infty} z_n e^{\la_n t}\phi_n\|^2 
\le \sum_{n=1}^{\infty} |(u_1(0),\phi_n) - (u_1(0),\phi_n)|^2  e^{2\la_n t}
\le 2 \sum_{n=1}^{\infty} \left\{|(u_1(0),\phi_n)|^2 +
  |(u_1(0),\phi_n)|^2\right\}  e^{2\la_n T} \le 2 (M_1^2 + M_2^2).
$
Thus \eqref{eq:series_z} forms a convergent series.}:
\be\label{eq:series_z}
z(t)=\sum_{n=1}^{\infty} z_n e^{\la_n t}\phi_n.
\ee
By \eqref{cond:secb} and \eqref{eq:series_z},
\be\label{eq:diff}
\|z(T)-z(s)\|^2 =\sum_{n=1}^{\infty}z_n^2(e^{\la_n T}-e^{\la_n s})^2
\le (K\tilde{\de})^2,
\ee
where $\tilde{\de}=2\de$. By the mean value theorem there exists 
$s_n$ ($s<s_n<T$) such that 
$(e^{\la_n T}-e^{\la_n s})=\la_n(T-s)e^{\la_n s_n}$  for each $n$. 
Therefore, utilizing $s<s_n$, from \eqref{eq:diff} it follows that
\be\label{eq:diff2}
\sum_{n=1}^{\infty}z_n^2\la_n^2(T-s)^2 e^{2\la_n s}\le (K\tilde{\de})^2.
\ee
By dividing both sides of \eqref{eq:diff2} by $(T-s)^2$ and 
replacing $\sum_{n=l+1}^{\infty}z_n^2\la_n^2 e^{2\la_n s}$
by $\sum_{n=l+1}^{\infty}z_n^2 e^{2\la_n s}$, a
rearrangement yields
\be\label{eq:diff3}
\sum_{n=l+1}^{\infty}z_n^2 e^{2\la_n s} \le 
\left(\frac{K\tilde{\de}}{T-s}\right)^2
- \sum_{n=1}^{l}z_n^2 \la_n^2 e^{2\la_n s}.
\ee
Finally by adding both sides of \eqref{eq:diff3} by 
$\sum_{n=1}^{l}z_n^2 e^{2\la_n s}$, 
\eqref{eq:series_z} implies that
\bes
\|z(s)\|^2=\sum_{n=1}^{\infty}z_n^2 e^{2\la_n s}\le
\left(\frac{K\tilde{\de}}{T-s}\right)^2 +
\sum_{n=1}^{l} z_n^2(1-\la_n^2)e^{2\la_n s}.
\ees
This proves the lemma.
\end{proof}

\begin{remark}
If $l=0$, we interpret $\sum_{n=1}^l$ as $0$.
\end{remark}
Now we are in a position to prove \thmref{thm:secb}.

\begin{proof}{\rm [of \thmref{thm:secb}]}
Let $z(t)=u_1(t)-u_2(t)$ and note that (see \cite{agmon, lin-backward})
\be\label{zt}
\|z(t)\|\le \|z(T)\|^{t/T}\|z(0)\|^{1-t/T}, ~t\in[0,T].
\ee
Let $\tilde{\de}=2\de$.
By \eqref{cond:secb} and the triangle inequality, we have
\be\label{here:secb}
\|z(0)\|\le \tilde{\de} \quad\mbox{ and }\quad 
\|z(T)-z(s)\|\le K\tilde{\de}.
\ee
By \eqref{zt}, the triangle inequality, and \eqref{here:secb},
we have
\be\nonumber
\|z(t)\| &\le& (\|z(T)-z(s)\|+\|z(s)\| )^{t/T} \|z(0)\|^{1-t/T}\\
         &\le& (K\tilde{\de} + \|z(s)\|)^{t/T} \tilde{\de}^{1-t/T},\qquad ~t\in[0,T].
\label{rewrite}
\ee
Due to \lemref{lem:bound}, there exists an upper bound
$\zeta_1>0$ for $K\tilde{\de}+\|z(s)\|$ such that
\be\label{abc}
K\tilde{\de}+\|z(s)\| \le \zeta_1. 
\ee
Then, \eqref{rewrite} implies that
\be\label{rewrite2}
\|z(s)\| \le \zeta_1^{s/T}\tilde{\de}^{1-s/T}. 
\ee
If $\zeta_1$ is a sharp estimate such that 
$\zeta_1 = K\tilde{\de}+\zeta_{1}^{s/T}\tilde{\de}^{1-s/T},$ 
then ${\zeta_1}/{\tilde{\de}}$ is the
unique root of \eqref{eq:alge}; in this case, set $\Lambda =
{\zeta_1}/{\tilde{\de}}.$
Otherwise, we have a successive estimate for $\zeta_2$ 
such that
\[
 K \tilde{\de}+\|z(s)\|\le  \zeta_2 \le
 K\tilde{\de}+\zeta_{1}^{s/T}\tilde{\de}^{1-s/T}
\] by
using \eqref{rewrite2} in place of $\|z(s)\|$ in \eqref{abc}.
A division of this by $\tilde{\de}$ gives
\bes
{\zeta_2}/{\tilde{\de}}\le K+\left({\zeta_{1}}/{\tilde{\de}}\right)^{s/T}.
\ees
Again if $\zeta_2$ is a sharp estimate such that 
$\zeta_2 = K\tilde{\de}+\zeta_{2}^{s/T}\tilde{\de}^{1-s/T},$ 
then ${\zeta_2}/{\tilde{\de}}$ is the
unique root of \eqref{eq:alge}; in this case, set $\Lambda =
{\zeta_2}/{\tilde{\de}}.$ 
Otherwise, continue this iteration which results in
\bes
{\zeta_n}/{\tilde{\de}} \le K+\left({\zeta_{n-1}}/{\tilde{\de}}\right)^{s/T}
\ees
for $n=1,2,3,\cdots.$
Consequently, a standard fixed point iteration argument 
implies that $\frac{\zeta_n}{\tilde{\de}} \le \Lambda,$ 
where $\Lambda$ is the unique root of \eqref{eq:alge}. Thus,
invoking \eqref{rewrite}, we arrive at
\begin{eqnarray}\label{eq:proof3.13}
\|z(t)\|\le (\Lambda\tilde{\de})^{t/T}\tilde{\de}^{1-t/T}=
\Lambda^{t/T}\tilde{\de},
\end{eqnarray}
which completes the proof.
\end{proof}
\begin{remark}\label{bcd}
In \cite{carasso}, The constraint $\|u_j(T)\|\le M$ was used to bound $\zeta_1$
in \eqref{abc}.
However, with any $\zeta_1>0$, the limit point of the sequence
$\left(\zeta_n/\tilde\delta\right)$ is bounded by $\Lambda$ (see  
\lemref{lemma:gamma_ineq}), and thus the use of $M$ in \cite{carasso}
is not essential provided $\zeta_1$ is bounded by some constant,
which is proved in \lemref{lem:bound}.
\end{remark}
\begin{remark}
An emphasis has to be made: the constant $s^*\in(0,T)$, which depends in a
priori bound $M$ as well as $\delta$ and $K$, is not a requirement in 
\thmref{thm:secb}.
\end{remark}

\begin{remark}
Notice that in \eqref{eq:final} the right hand side
goes to infinity as $s\to T$. Thus, if $s=T$(it means we do not have the SECB
constraint),
we may not bound 
$\|z(s)\|$, and therefore \eqref{eq:thm_diff} may not be proved, {\it i.e.,}
the SECB constraint actually guarantees continuous dependence on data in place of 
$\|u_j(T)\|\le M$.
\end{remark}
\begin{remark}\label{rem:7}
For some $M$, $K>0$ and $0<\de\ll 1$, 
let $s^*$ be such that $\frac{M}{\de}=K+(\frac{M}{\de})^{s^*/T}.$
The following three cases should be considered:
\begin{enumerate}
\item
 In the case of the SECB constraint with $s=s^*$, the unique root $\Lambda$ of 
\eqref{eq:alge} becomes $\frac{M}{\de}$. Thus by \thmref{thm:secb}
we have $\|z(t)\|\le 2\Lambda^{t/T}\de=2M^{t/T}\de^{1-t/T}$ as in \eqref{eq:holder}.
This means that the SECB constraint is indeed so general to include
the case in which the F. John's bounded constraint $\|u_j(T)\|\le M$ is used.
\item In the case of $s<s^*$, one can use the SECB constraint which will lead to a
substantial improvement over the F. John's a priori estimate
\eqref{eq:holder}.
\item
However, in the case of  $s^*<s<T$, the stability estimate \eqref{eq:thm_diff}
might be worse than that \eqref{eq:holder}. 
\end{enumerate}
Combining these facts,
the knowledge of an a priori bound in \eqref{condition1} will be still useful for
an exact computation of $s^*$ that will guarantee a choice of $s<s^*$.
\end{remark}

\begin{remark}
It should be emphasized that 
even in the case that the a priori bound $M$ in \eqref{condition1} is not used and thus
$s^*$ is not exactly given, the SECB constraint \eqref{cond:secb} will provide
the stability estimate \eqref{eq:thm_diff} which guarantees the
continuous dependence of the solutions on the initial data up to the boundary $t=T$.
\end{remark}

\section{A constructive regularized solution}\label{sec4}
In this section we propose a new regularized solution
to backward parabolic problems
based on the observation in the previous section.

Let $g\in L^2(\O)$ be given initial data.
Suppose $0<\delta \ll 1$,  $K \gg 1$ and $s\in(0,T)$ are given. 
Then with the $\Lambda$ which is the unique root of \eqref{eq:alge},
choose an appropriate contour $\Gamma \subset \rho(-A)$. For instance,
following \cite{jleesheen-back, sst-para2}, for suitable $\gamma, \nu,$ and
$\sigma$, let
\begin{eqnarray}
\Ga&=&\left\{z=z(y)\, |~ z(y) = \gamma - \sqrt{\nu^2 + y^2} + i\sigma y,
\text{ with increasing }  y \text{ from } -\infty \text{ to }
+\infty\right\},\nonumber \\
 &&\Re(z(0))= \frac{\log\Lambda}{T},
\text{ where } \Lambda \text{ is the unique root of \eqref{eq:alge} }.
\label{eq:Gamma}
\end{eqnarray}
Although one does not have a precise information on the exact eigenvalues of
$-A$,
there will be a finite number of eigenvalues which are strictly less than $\frac{\log\Lambda}{T}.$
Let $\lambda_1,\cdots,\lambda_N$ be all such eigenvalues.
Also, denote 
\[
\Phi_\G = \operatorname{span}\{\phi_1,\cdots,\phi_N\} \subset L^2(\O)
\]
and let $\Pi_\G: L^2(\O)\rightarrow \Phi_\G$ be the $L^2(\O)$-projection.
For $u_{j0}\in L^2(\O)$, define
\be
u^{\Ga,u_{j0}}(t)=\frac{1}{2\pi \i}\int_{\Ga}e^{zt}v(z) \,dz,
\label{eq:regul_sol}
\ee
where $v=v(z)=v(\cdot,z) \in H^1_0(\O)$ is the unique solution to
\be\label{eq:freq}
zv+Av=u_{j0}, \quad z\in \rho(-A).
\ee
Formally, the $u^{\G,u_{j0}}$ can be written as
\be\label{eq:regul_sol2}
u^{\Ga,u_{j0}}(t)=\frac{1}{2\pi \i}\int_{\Ga}e^{zt} (zI+A)^{-1} u_{j0} \,dz.
\ee
Observe that $u_j = u^{\G,u_{j0}}$ satisfies
\be
\frac{\p u_j}{\p t} + Au_j &=& 0 \quad\text{on }\O\times (0,T).\label{eq:problem2}
\ee
Define the class of  {\it new regularized solutions} by
\be
\left\{u_j= u^{\Ga,u_{j0}}\,|\, u^{\Ga,u_{j0}} \text{ is defined by }
  \eqref{eq:regul_sol} \text{ and } \eqref{eq:freq};\, \|u_{j0}-g\|\le \de,\, \|u_j(T)-u_j(s)\| \le K\de \right\}.
\label{eq:new_regul_sol}
\ee

Notice that the integrands of two Cauchy integrals \eqref{eq:regul_sol}
and \eqref{eq:regul_sol2} can be written as the infinite series
$\sum_{k=1}^\infty \frac{e^{zt}}{z-\la_k}(u_{j0},\phi_k)\phi_k.$
Among them only a finite number of terms from $k=1$ to $k=N$ are to the
left of the contour $\G$, and the rest of infinite terms are analytic in the
left half plane. Thus the Cauchy integrals \eqref{eq:regul_sol}
and \eqref{eq:regul_sol2}  are convergent. Indeed, we have
the following spectral reprentation formula:
\begin{proposition}\label{prop:spect}
The spectral representation of $u^{\G,u_{j0}}(t)$ is given by
\bes
u^{\Ga,u_{j0}}(t)=\sum_{k=1}^N e^{\la_kt}(u_{j0},\phi_k)\phi_k.
\ees
\end{proposition}
\begin{proof}
By taking the $L^2(\O)$-inner product of both sides of \eqref{eq:freq}
against $\phi_k$,
one gets a spectral representation of $v(z)$ by
\be\label{eq:v_expansion}
v(z)=\sum_{k=1}^\infty \frac{1}{z-\la_k}(u_{j0},\phi_k)\phi_k.
\ee
Since $N$ is the largest integer such that $\la_N < \frac{\log\Lambda}{T}$
and $\sum_{k=N+1}^\infty\frac{1}{z-\la_k}(u_0,\phi_k)\phi_k$
is analytic in the half plane left to $\Ga$,
by incorporating \eqref{eq:regul_sol},  
\eqref{eq:v_expansion}, and Cauchy's integral theorem, we have
\bes
u^{\Ga,u_{j0}}(t)&=&\frac{1}{2\pi \i} \sum_{k=1}^{\infty}\int_\Ga \frac{e^{zt}}
{z-\la_k}(u_{j0},\phi_k)\phi_k\,dz \\&=& 
\frac{1}{2\pi \i}\sum_{k=1}^{N}\int_\Ga \frac{e^{zt}}{z-\la_k}
(u_{j0},\phi_k)\phi_k \,dz
= \sum_{k=1}^N e^{\la_kt}(u_{j0},\phi_k)\phi_k.
\ees
This completes the proof.
\end{proof}

Due to the definition of $u^{\G,u_{j0}}$, the same stability estimate as 
\eqref{eq:thm_diff} follows. 
\begin{corollary}\label{thm:main_secb}
Let $u^{\Ga,u_{10}}$ and $u^{\Ga,u_{20}}$ be new regularized solutions. Then,
\begin{eqnarray}\label{eq:thm45}
\| u^{\Ga,u_{10}}(t) - u^{\Ga,u_{20}}(t) \| \le 2 \Lambda^{t/T} \de, \quad
t \in [0,T].
\end{eqnarray}
\end{corollary}
\begin{proof}
Since
\[
\| u^{\Ga,u_{10}}(0) - u^{\Ga,u_{20}}(0) \| = 
\| \Pi_\G u_{10} - \Pi_\G u_{20} \| \le 
\| u_{10} -  u_{20} \| \le \delta, 
\]
the estimate \eqref{eq:thm45} is an immediate consequence of \thmref{thm:secb}.
\end{proof}

\begin{remark}\label{rem:4.3}
The implementation of the above regularized solution can be given as follows.
\begin{enumerate}
\item Given $\delta, K, s,$ solve for $\Lambda$ satisfying \eqref{eq:alge}.
\item Choose a contour $\G$ as in \eqref{eq:Gamma}.
\item Deform the contour $\G$ as a hyperbola, parabola, or Talbot contour with
parameter to be the imaginary part. (See, for more details, \cite{sst-para2, weid-tref-07, weid-tref-schm-06}.)
\item Represent the infinite contour as a graph of a function on a finite interval, on which choose
composite trapezoidal points. (Again refer to \cite{sst-para2, weid-tref-07, weid-tref-schm-06}.)
\item Solve a set of complex-valued, Helmholtz-type problems \eqref{eq:freq}
for such finite number of contour points, by any kind of space discretization methods (e.g. finite
element method, spectral method, etc.)
\item Take a discrete sum of such solution to approximate the contour representation of the solution given in \eqref{eq:regul_sol}.
\end{enumerate}
This type of procedures, called ``Laplace transformation method for parabolic
problems'', have been proposed and analyzed for forward parabolic problems,
in \cite{sst-para, sst-para2, gavril01, gavril02, gavril04, gavril05, jleesheen-laplace, marban-palencia,
  marban-palencia2, mclean-thomee, thomee-ijnam, weid-06, weid-tref-07, weid-tref-schm-06},
integro-differential equations \cite{kwonsheen-memory, mclean-sloan-thomee}
and backward parabolic problems \cite{jleesheen-back, jleesheen-denoising}.
\end{remark}

\section{Numerical examples}\label{sec5}
In this section, we try to find regularized solutions
which are members of 
the class defined in \eqref{eq:new_regul_sol},
and illustrate the behavior of them.
We construct the solutions following the steps in \rmkref{rem:4.3}
for the following backward parabolic problem:
\bes
u_t+cu_{xx}&=&0 \quad\text{on }\O\times (0,T),\\
u &=&0 \text{ on } \partial \Om,
\ees
where $\Om=(0,\pi), T=4,$ and $c=1/32$.
Let the piecewise linear solution at $t=T$,
which should be sought, given by
\bes
u(\cdot, T)=\left\{ \begin{array}{lr}
\frac{16}{\pi}x, & 0\le x\le\frac{\pi}{4},\\
-\frac{16}{\pi}x +8, & \frac{\pi}{4}\le x \le \frac{\pi}{2},\\
0, & \mbox{ otherwise}. \end{array}\right.
\ees
The following truncated series solutions 
\be\label{eq:ser_sol}
\tilde u(x,t) = \sum_{k=1}^{1000} \frac1{c (k \pi)^2} \left[ 2 \sin\left(\frac{3 k\pi}{4}\right)
             - \sin \left(\frac{k\pi}{2}\right)  - \sin (k\pi) \right]
 \sin( kx )  e^{ -c k^2 (T-t)}
\ee
are used for the reference solutions $u_0(x), u(x,T/4), u(x,T/2),$ and $u(x,3T/4).$

Let $s=3.8$. Based on the information on $\tilde u_0$, we fix 
$K=0.142/\de$ for given $\de$.

To find solutions $u^{\Ga,u_{j0}}$ in the class \eqref{eq:new_regul_sol},
initial data $u_{j0}$ are generated by perturbing the above truncated series
solution $\tilde u_0(x)= \tilde u(x,0)$ given by \eqref{eq:ser_sol}
using the Fortran 90/95 intrinsic subroutine $random\_number(\cdot)$ 
so that $\|u_{j0}-\tilde u_0\|<\de$ for $\de=10^{-4},10^{-3},10^{-2}$.
Then the solutions $u^{\Ga,u_{j0}}(x,t)$ are approximated by using the standard
piecewise linear finite element with sufficiently many element (1024 meshes)
and the Laplace transformation method on the deformed contour
$
\Ga =\left\{ z(y) = \gamma(\delta,s,K)-\sqrt{\nu^2 + y^2} + \i y, \quad y\in [-\infty, \infty]\right\}
$
in \eqref{eq:alge} is used, where $\nu=0.5$ and $\gamma(\delta,s,K)\simeq
2.583, 2.067, 1.074$ for $\delta = 10^{-4}, 10^{-3}, 10^{-2}$,
respectively. In all experiments 160 number of contour points $z_j$'s are 
chosen which will be sufficient to circumvent the numerical errors in time
discretization.

\tabref{table1} shows the $L^2$ errors of $u^{\Ga,u_{j0}}$ to the reference solutions
and the theoretical upper bound values $2\Lambda^{t/T}\de$,
as given in Corollary \ref{thm:main_secb}
for each $\delta = 10^{-4},10^{-3},10^{-2}$. Notice that the
$L^2$ errors of $u^{\Ga,u_{j0}}$ are much smaller than the predicted
bounds for all cases. To check $u^{\Ga,u_{j0}}$ is in the
desired class, we calculated $\|u^{\Ga,u_{j0}}(T)-u^{\Ga,u_{j0}}(s)\|$
values, which are $0.0568,0.0623$ and $0.208$ for 
$\de=10^{-4},10^{-3},10^{-2}$, respectively. 
They are less than $K\de=0.142$ except for the case of $\de=10^{-2}$.
When $\de=10^{-2}$, we tried to generate a solution in the class
by introducing several but not so many enough number of different random
noises to no avail.
However, even though the initial data is not in the class if it is near,
the $L^2$ errors are smaller than
the expected bounds, and thus the quality of the solutions shown 
in Figure 1 are acceptable for all cases including when
$\de=10^{-2}$.

\begin{figure}[ht]
\label{fig1}
\begin{center}
\epsfig{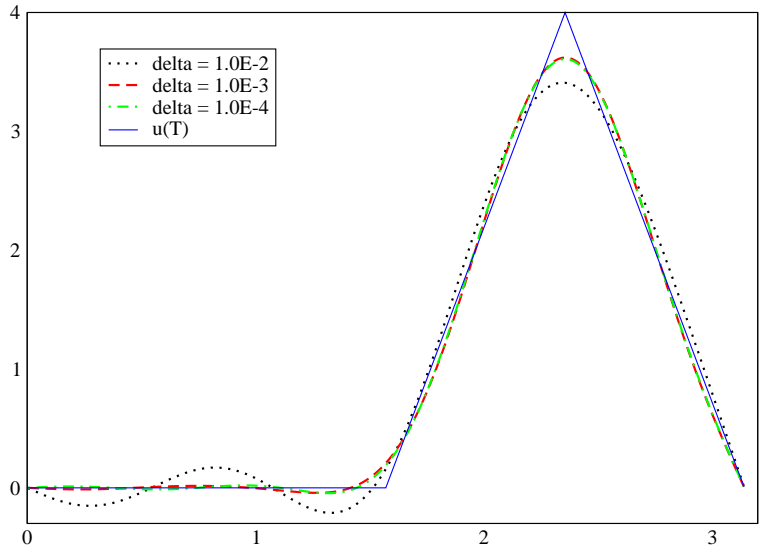}
\vspace{1cm}
\caption{Exact and computed solutions at $t=T$ with various $\de$'s and 
$s=3.8$.}
\end{center}
\end{figure}

\begin{table}[tbhp]
\begin{center}
\begin{tabular}{|c|c|c|c|c|c|c|} \hline
&\multicolumn{2}{|c|}{$L^2$-errors with $\delta = 10^{-4}$}
&\multicolumn{2}{|c|}{$L^2$-errors with $\delta = 10^{-3}$}
&\multicolumn{2}{|c|}{$L^2$-errors with $\delta = 10^{-2}$} \\ \cline{2-7}
\rb{$t$} & Computed & Predicted & Computed & Predicted& Computed & Predicted\\
\hline
T/4 & 4.25E-05 & 1.61E-03 & 2.96E-04 & 9.59E-03 & 7.82E-03&5.86E-02 \\
\hline
T/2 & 3.18E-04 & 1.29E-02 & 1.20E-03 & 4.59E-02 & 2.39E-02&1.71E-01\\
\hline
3T/4 & 4.65E-03 & 1.04E-01 & 6.85E-03 & 2.20E-01 &7.38E-02&5.02E-01\\
\hline
T & 1.48E-01 & 8.33E-01 & 1.50E-01 & 1.06E-00 & 2.72E-01&1.47E-00\\
\hline
\end{tabular}
\vspace{0.1cm}
\caption{$L^2$ errors of our regulared solutions and those of
prediced in the theory are shown for $\delta = 10^{-4},10^{-3},10^{-2}$. 
Here $s=3.8$.}
\label{table1}
\end{center}
\end{table}

\tabref{table2} and Figure 2 are for $s=3.9$. The
$L^2$ errors of $u^{\Ga,u_{j0}}$ are again smaller than the predicted
bounds for all cases. With $s =3.9$, the
$\|u^{\Ga,u_{j0}}(T)-u^{\Ga,u_{j0}}(s)\|$ values are $0.0568, 0.0623$ and $0.2110$ for 
$\de=10^{-4},10^{-3},10^{-2}$, respectively. 
We observe that they are less than $K\de=0.084$ except for $\de=10^{-2}$.
As in the previous case with $s=3.8$, for $\de=10^{-2}$, 
we tried to generate a solution in the class
by introducing different random noises, but it failed.
It also should be noted that the $L^2$ errors are also smaller than
the expected bounds, which implies that the quality of the solutions shown 
in Figure 2 are acceptable for all cases including when
$\de=10^{-2}$.

\begin{figure}[ht]
\label{fig2}
\begin{center}
\epsfig{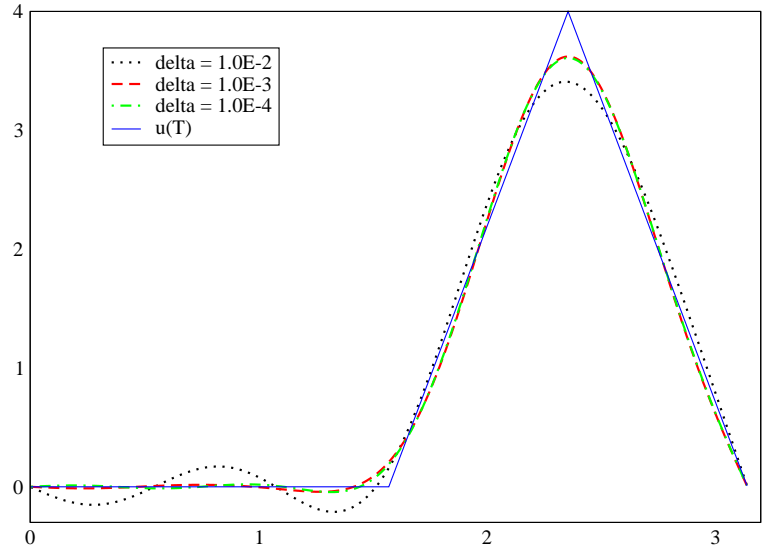}
\vspace{1cm}
\caption{Exact and computed solutions at $t=T$ with various $\de$'s and
  $s=3.9$.}
\end{center}
\end{figure}

\begin{table}[tbhp]
\begin{center}
\begin{tabular}{|c|c|c|c|c|c|c|}
\hline
 & \multicolumn{2}{c|}{$L^2(\O)$ errors with  $\de = 10^{-4}$} & 
\multicolumn{2}{|c|}{$L^2(\O)$ errors with $\de = 10^{-3}$} & 
\multicolumn{2}{|c|}{$L^2(\O)$ errors with $\de = 10^{-2}$} \\
\cline{2-7}
\multicolumn{1}{|c|}{\rb{$t$}} & Computed & Predicted & Computed & Predicted& Computed & Predicted\\
\hline
T/4 & 4.25E-05&1.63E-03 & 2.96E-04& 9.78E-03&7.83E-03&6.00E-02\\
\hline
T/2 & 3.18E-04&1.33E-02 &1.20E-03&4.79E-02&2.39E-02&1.80E-01\\
\hline
3T/4 & 4.65E-03 & 1.09E-01&6.84E-03&2.34E-01&7.40E-02&5.39E-01\\
\hline
T & 1.48E-01 & 8.88E-01 & 1.50E-01 & 1.15E-00&2.72E-01&1.62E-00\\
\hline
\end{tabular}
\vspace{0.1cm}
\caption{$L^2$ errors of our regulared solutions and those of
prediced in the theory are shown for $\delta = 10^{-4},10^{-3},10^{-2}$. 
Here $s=3.9$.}
\label{table2}
\end{center}
\end{table}

\section*{Acknowledgements}
The authors wish to thank the anonymous referees for their helpful
comments, which are reflected in the significantly-improved final version.
JL  was supported by the Korea Research Foundation Grant (KRF-2007-331-C00051) 
and
DS was supported in part by the Korea Research Foundation Grant
(KRF-2006-070-C00014) and Korea Science and Engineering Foundation (KOSEF
R01-2005-000-11257-0), KOSEF R14-2003-019-01002-0(ABRL), and the Seoul R\&BD Program.

\bibliographystyle{plain}

\end{document}